\documentclass[11pt]{amsart}
\usepackage[utf8]{inputenc}
\usepackage{color}
\usepackage{tikz}
\usetikzlibrary{patterns}
\usepackage[margin=2.5cm]{geometry}
\usepackage{verbatim}
\usepackage[colorinlistoftodos,prependcaption,textsize=small]{todonotes}
\usepackage{amssymb}
\usepackage{url}
\usepackage{enumerate}

\def\peb[#1]{{\left\lfloor #1\right\rfloor}}
\def\pe[#1]{{\left\lceil #1\right\rceil}}
\newcommand{\cone}{\operatorname{L}_{\mathbb{Q_+}}}

\title[Stratified sets of atoms]{Factorization in monoids by stratification of atoms and the Elliott Problem}
\author{Pedro A. Garc\'ia-S\'anchez}
\address{Departamento de \'Algebra and IEMath-GR, Universidad de Granada, E-18071 Granada, Espa\~na}
\email{pedro@ugr.es}
\author{Ulrich Krause}
\address{Universit\"at Bremen, Fachberich Mathematik/Informatik, D-28359 Bremen, Germany}
\email{ukrause@uni-bremen.de}
\author{David Llena}
\address{Departamento de Matem\'{a}ticas, Universidad de Almeria, E-04120 Almer\'ia,  Espa\~na}
\email{dllena@ual.es}
\thanks{The first and third authors are supported by the project MTM2017-84890-P, which is funded by Ministerio de Econom\'{\i}a, Industria y Competitividad and Fondo Europeo de Desarrollo Regional FEDER, and by the Junta de Andaluc\'{\i}a Grant Number FQM-343.}

\thanks{The authors would like to express their gratitude to the referee for their detailed and exhaustive report.}
\date{\today}

\newtheorem{theorem}{Theorem}
\newtheorem{proposition}[theorem]{Proposition}
\newtheorem{lemma}[theorem]{Lemma}
\newtheorem{corollary}[theorem]{Corollary}

\theoremstyle{remark}
\newtheorem{example}{Example}
\newtheorem{remark}[theorem]{Remark}

\DeclareMathOperator{\Ap}{Ap}
\DeclareMathOperator{\supp}{supp}

\begin{document}

\begin{abstract}
	In  an additive factorial monoid each element can be represented as a linear combination of irreducible elements (atoms) with uniquely determined coefficients running over all natural numbers. 
	
	In this paper we develop for a wide class of non-factorial monoids a concept of stratification for atoms which allows to represent each element as a linear combination of atoms where the coefficients are uniquely determined when restricted in a particular way. This wide class includes inside factorial monoids and in particular simplicial affine semigroups. In the latter case the question of uniqueness is related to a problem studied by E. B. Elliott in a paper from 1903.
	
	For the monoid of all nonnegative solutions of a certain linear Diophantine equation in three variables, Elliott considers ``simple sets of solutions'' (atoms of the monoid) and looks for a method that gives ``every set once only''. We show that for simplicial affine semigroups in two dimensions a stratification is always possible, which answers to Elliott's problem also for the cases he left open.
	
	The results in this paper on the stratification of atoms for monoids in general may be seen also as an answer to a ``generalized Elliott problem''. 
\end{abstract}
\maketitle

\section{Introduction}

In \cite{elliot}, Elliott proposed the following problem: given the linear Diophantine equation $ax+by=cz$, with $a$, $b$ and $c$ positive integers, is it possible to represent all nonnegative integer solutions in a unique manner by means of a ``simple set of solutions''? He was able to describe parametrically the solutions for small $c$ with the help of generating functions: every solution to $ax+by=cz$ had a unique expression as a linear combination of some basic solutions with some coefficients free and others bounded or fulfilling some linear constraints. In \cite{GSKL} we gave a possible approach to this problem based on Ap\'ery sets, but still the solution proposed was not in general in the spirit of \cite{elliot}. In the present manuscript we provide a solution that generalizes the idea employed in \cite{elliot}, which is based on a stratification of the Hilbert basis of the set of solutions by successive elimination of extremal rays. We will see this approach in Section~\ref{sec:motivation}, inspired in the primary decomposition of an element in a numerical semigroup by using a sequence of incremental Ap\'ery sets. We will show that the monoids considered by Elliott in his paper admit a stratification of its sets of atoms, giving in this way a solution to the problem of expressing every nonnegative integer solution to the equation $ax+by=cx$ uniquely in terms of a ``simple set of solutions''. 

The idea of stratification proposed in Section~\ref{sec:motivation} relies on the concept of extremal ray, which appears in the literature under different names depending on the scope it is considered. Section~\ref{sec:extraction} is devoted to the study of pure and strong atoms, and to relate these to notions to the concept of absolutely irreducible element, completely fundamental element, extreme rays in affine semigroups, and elementary atoms in block monoids. In that section we also study some properties of extraction monoids of finite type that will enable to relate Apéry sets with extraction degrees. As an example we will see that this relationship recovers the parametrizations given by Elliott in \cite{elliot}. We will show that several interesting and well studied families of monoids are extraction monoids of finite type.

The last section generalizes the idea of stratification given in the motivation section (Section~\ref{sec:motivation}). We introduce the concept of set linked to a monoid, which in some sense extrapolates the idea of extremal rays in an affine semigroup to cancellative and reduced monoids. When a set $Q$ is linked to a monoid $M$, then every element in the monoid can be expressed as a combination of elements in $Q$ plus an element in the Ap\'ery set of $Q$ in $M$. We characterize when this way of expressing elements is unique, which actually serves as initial step of a stratification of the set of atoms. This characterization is inspired in (and generalizes at the same time) the decomposition that was already observed in \cite[Theorem~1.5]{c-m} for Cohen-Macaulay simplicial affine semigroups. We prove that the basis of an inside factorial monoid is always linked to the monoid, and if in addition the monoid is root-closed, then the uniqueness mentioned above holds. Several examples are given for which a stratification of the set of atoms can be performed.

\section{Some notation}

We tried to keep this section as small as possible, and when needed we will introduce new concepts and definitions. 

Let $(M,+)$ be an additive commutative monoid. We say that $M$ is \emph{cancellative} if  $a+b=a+c$ implies $b=c$ for all $a,b,c\in M$. An element $u\in M$ is a \emph{unit} if there exists $v\in M$ with $u+v=0$. The monoid $M$ is \emph{reduced} if its only unit is its identity element, denoted by $0$. 

An non-unit element $a\in M$ is an \emph{atom} if whenever $a=b+c$ for some $b,c\in M$, either $b$ or $c$ is a unit. 

Monoids considered in this manuscript are reduced and cancellative.  In this setting, $a\in M\setminus\{0\}$ is an atom if it cannot be expressed as a sum of two non-zero elements of $M$. These elements are also known in the literature as irreducibles. We denote the set of atoms of $M$ by $\mathcal{A}(M)$. We say that $M$ is \emph{atomic} if $M=\langle \mathcal{A}(M)\rangle$, that is, every element in $M$ can be expressed as a sum of finitely many atoms \cite{g-hk}.

As we are assuming that $M$ is cancellative, $M$ can naturally be embedded into its quotient group, which we denote by $\operatorname{G}(M)$, and we can identify its elements with $a-b$ with $a,b\in M$.

Let $M$ be a monoid and let $X$ be a subset of $M$. The set $I=X+M=\{x+m \mid x\in X, m\in M\}$ is an ideal of $M$ ($I+M\subseteq I$). We define the \emph{Apéry set} of $X$ in $M$ as 
\[
\Ap(M,X)=M\setminus (X+M).
\]
If $X=\{x\}$, we will write $\Ap(M,x)$ for $\Ap(M,\{x\})$.

\section{Motivation}\label{sec:motivation}

Given a numerical semigroup $S$ and a nonzero element $m\in S$, every element $x\in S$ can be expressed uniquely as 
\begin{equation}\label{eq:ap-ns}
	x=km+w
\end{equation}
for some nonnegative integer $k$ and $w\in \Ap(S,m)$ (see for instance \cite[Lemma~2.6]{RG}). Let $n_1,\ldots,n_e$ be the minimal generators of $S$. We can express $x=k_1n_1+w_1$ with $w_1\in \Ap(S,n_1)$ and $k_1\in\mathbb{N}$, and $k_1$ and $w_1$ are unique. We can then write $w_1=k_2n_2+w_2$, for some (unique) $k_2\in \mathbb{N}$ and $w_2\in \Ap(S,n_2)$. Notice that $w_2$ is also in $\Ap(S,n_1)$, since $w_1\in \Ap(S,n_1)$. Hence, $x=k_1n_1+k_2n_2+w_2$ for some unique $k_1,k_2\in \mathbb{N}$ and $w_2\in \Ap(S,\{n_1,n_2\})$. We can repeat this process until $x=k_1n_1+\dots+k_{e-1}n_{e-1}+k_en_e+w_e$ with $k_1,\dots,k_e\in \mathbb{N}$ and $w_e\in \Ap(S,\{n_1,\dots,n_e\})$. But as $\{n_1,\ldots,n_e\})$ is a generating system of $S$, we have that $\Ap(S,\{n_1,\dots,n_e\})=\{0\}$. Thus for every $x\in S$, there exists unique $k_1,\ldots,k_e\in\mathbb{N}$ such that
\begin{itemize}
	\item $x=k_1n_1+\dots+k_e n_e$,
	\item $k_{i+1}n_{i+1}+\dots+k_en_e\in \Ap(S,\{n_1,\ldots,n_i\})$ for all $i\in\{1,\dots, e\}$.
\end{itemize}  
The expression $k_1n_1+\dots+k_en_e$ is known as the \emph{primary representation} of $x$ with respect to the arrangement $n_1,\dots,n_e$ of the minimal generators of $S$.

Elliot was concerned with a unique representation of the nonnegative integer solutions of $ax+by=cz$. The monoid $N=\{(x,y,z)\in \mathbb{N}^3 \mid ax+by=cz\}$, with $a$, $b$ and $c$ positive integers admits a unique minimal generating system known as the \emph{Hilbert basis} of $N$. This monoid is isomorphic to $M=\{(x,y)\in \mathbb{N}^2\mid ax+by\equiv 0\pmod c\}$. A way to parametrize solutions in terms of Apéry sets was given in \cite{GSKL}. Let us see with an example if we can produce something analogous to primary representations on numerical semigroups.

Observe that $M$ (and $N$) are \emph{full}, that is, for all $x,y\in M$, $x-y\in\mathbb{N}^2$ if and only if $x-y\in M$, or in other words, $\operatorname{G}(M)\cap\mathbb{N}^2=M$, where $\operatorname{G}(M)=\{a-b\mid a,b\in M\}$ is the group spanned by $M$ in $\mathbb{Z}^2$.

\begin{example}\label{ElliottExample}
	Let  $M$ be the set of non-negative integer solutions of $x+2y\equiv 0 \pmod 7$. The Hilbert basis of $M$ (there are many ways to compute this Hilbert basis) is \[H=\{(0,7),(1,3),(3,2),(5,1),(7,0)\}.\] This example is not covered by \cite[Corollary 5]{GSKL}, as $H$ has three atoms not in $\{(0,7),(7,0)\}$. More precisely, we have that $(6,4)=(1,3)+(5,1)=(3,2)+(3,2)$, that is, there exist two different ways to write $(6,4)$ as combination of elements from $H$.

	We want to find for any $(x,y)\in M$ a unique representation of $(x,y)$ as a combination of elements in $H$ subject to some Apéry restrictions. In a numerical semigroup we may choose any nonzero element and get a representation of the form \eqref{eq:ap-ns}. Set $u_1=(7,0)$ and $v_1=(0,7)$, which clearly span the same cone as $M$: $\cone(M)=\cone(\{u_1,v_1\})=\{au_1+bv_1\mid a,b\in \mathbb{Q}, a\ge 0,b\ge 0\}$. The monoid $M$ is Cohen-Macaulay (\cite{rgs-full}, we will come back to this later), and so we can use \cite[Lemma~1.4 and Theorem ~1.5]{c-m} to find unique $k_1$ and $l_1$ in $\mathbb{N}$, and $w_1\in \Ap(M,\{(7,0),(0,7)\})$ such that $(x,y)=k_1(7,0)+l_1(0,7)+w_1$. 

	Lemma~13 in \cite{GSKL} states that
	\[
	\Ap(M,\{u_1,v_1\})=\{(0,0),(1,3),(2,6),(3,2),(4,5),(5,1),(6,4)\}\subset [0,7)\times [0,7).
	\]
	Now, if $w_1\neq (0,0)$, we want to find a way to express it uniquely with respect to some elements in $\Ap(M,\{u_1,v_1\})$. Notice that all the elements in $\Ap(M,\{u_1,v_1\})$ are in the cone spanned by $(1,3)$ and $(5,1)$. So let us denote $u_2=(1,3)$ and $v_2=(5,1)$. Observe that 
	\[\begin{array}{rcl}
	(1,3) & = & 1\times (1,3)+ 0\times(5,1), \\
	(2,6) & = & 2\times (1,3)+ 0\times (5,1), \\
	(5,1) & = & 0\times (1,3)+ 1\times (5,1), \\
	(6,4) & = & 1\times (1,3)+ 1\times (5,1), \\
	\end{array} 
	\]
	and that these expressions are unique (the monoid $\langle u_2,v_2\rangle$ is factorial). Observe that $(3,2)\in \Ap(M,\{u_1,v_1,u_2,v_2\})$, and $(4,5)=1\times (1,3)+(3,2)$. Notice that $\Ap(M,\{u_1,v_1,u_2,v_2\})=\{(0,0),(3,2)\}$. Take $u_3=(3,2)$.
		
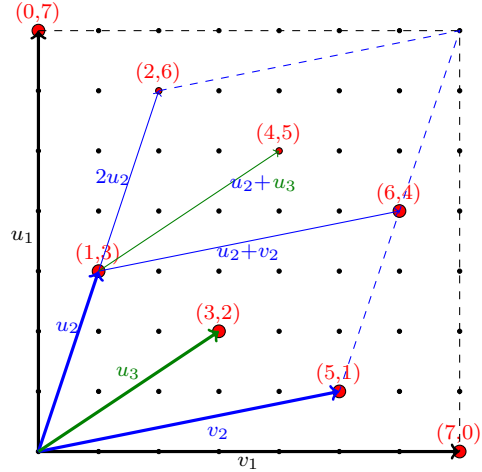
\begin{center}
	\begin{tikzpicture}[scale=0.4pt]
		\foreach \b in {0, 2,...,14}
		\foreach \a in {0, 2,...,14}
		\filldraw[fill=black] (\a,\b) circle (2pt);
		\filldraw[fill=red] (0,14) circle (6pt); \node at (0,14.5) {{\color{red}$^{(0,7)}$}};	
		\filldraw[fill=red] (2,6) circle (6pt); \node at (2,6.5) {{\color{red}$^{(1,3)}$}};	
		\filldraw[fill=red] (4,12) circle (3pt); \node at (4,12.5) {{\color{red}$^{(2,6)}$}};	
		\filldraw[fill=red] (6,4) circle (6pt); \node at (6,4.5) {{\color{red}$^{(3,2)}$}};	
		\filldraw[fill=red] (8,10) circle (3pt); \node at (8,10.5) {{\color{red}$^{(4,5)}$}};	
		\filldraw[fill=red] (10,2) circle (6pt); \node at (10,2.5) {{\color{red}$^{(5,1)}$}};	
		\filldraw[fill=red] (12,8) circle (6pt); \node at (12,8.5) {{\color{red}$^{(6,4)}$}};	
		\filldraw[fill=red] (14,0) circle (6pt); \node at (14,0.5) {{\color{red}$^{(7,0)}$}};	
		\draw[->,very thick] (0,0) -- (0,14); \draw[-,dashed] (0,14) -- (14,14) -- (14,0);	
		\draw[->,very thick] (0,0) -- (14,0); 	
		\node at (-0.5,7) {{\color{black}$^{u_1}$}}; \node at (7,-0.5) {{\color{black}$^{v_1}$}};		
		\draw[->,blue,very thick] (0,0) -- (2,6); \draw[->,blue] (2,6) -- (4,12); 	\draw[-,dashed, blue] (4,12) -- (14,14) -- (10,2); 
		\draw[->,blue,very thick] (0,0) -- (10,2);\draw[->,blue] (2,6) -- (12,8); 
		\node at (1,4) {{\color{blue}$^{u_2}$}}; 	\node at (2.5,9) {{\color{blue}$^{2u_2}$}}; \node at (6,0.5) {{\color{blue}$^{v_2}$}}; \node at (7,6.5) {{\color{blue}$^{u_2+v_2}$}};
		\draw[->,green!50!black,very thick] (0,0) -- (6,4); 	\draw[->,green!50!black] (2,6) -- (8,10);
		\node at (3,2.5) {{\color{green!50!black}$^{u_3}$}}; \node at (7.5,8.8) {{\color{blue}$^{u_2+}$}{\color{green!50!black}$^{u_3}$}};   	
	\end{tikzpicture}
\end{center}
	
	Set $H_1=\{u_1,v_1\}$, $H_2=\{u_2,v_2\}$ and $H_3=\{u_3\}$. Then, it follows that for every $(x,y)\in M$ there exists unique $h_i\in \langle H_i\rangle$, $i\in\{1,2,3\}$, such that 
	\begin{itemize}
		\item $(x,y)=h_1+h_2+h_3$,
		\item $h_{i+1}+\dots+h_3\in \Ap(M,H_1\cup\dots\cup H_i)$.
	\end{itemize}
	The role of $n_i$ in the numerical semigroup $S$ is now played by $H_i$, and the expression $k_in_i$ is now $k_iu_i+l_iv_i\in \langle H_i\rangle$, which is a factorial monoid, and thus the $k_i, l_i\in \mathbb{N}$ are unique.
\end{example}

We see next that this example is not exceptional; every Hilbert basis of a full affine subsemigroup of $\mathbb{N}^2$ will have a similar decomposition.

Observe that every affine semigroup $M$ contained in $\mathbb{N}^2$ with $\operatorname{rank}(\operatorname{G}(M))=2$ is simplicial, that is, there exists two linearly independent $u,v\in M$ such that $\cone(M)=\cone(\{u,v\})$. We will say that $u$ and $v$ are extremal rays of $M$ (though we will come back to this concept later).

Let $M$ be a full affine semigroup contained in $\mathbb{N}^2$ and with $\operatorname{rank}(\operatorname{G}(M))=2$. Let $H$ be the Hilbert basis of $M$. Then according to \cite[Theorem~10]{rgs-full}, there are two minimal generators in $H$ of the form $(\alpha_1,0)$ and $(0,\alpha_2)$. Moreover, $\Ap(M,\{(\alpha_1,0),(0,\alpha_2)\}) \subseteq [0,\alpha_1-1]\times[0,\alpha_2-1]$. Set $H_1=\{(\alpha_1,0),(0,\alpha_2)\}$. It follows that $\cone(M)=\cone(H_1)$. Now we consider $H_2\subseteq H\setminus H_1$ to be minimal with the condition that $H\setminus H_1\subset \cone(H_2)$. If the cardinality of $H$ is greater than three, then $H_2$ will have two elements. We now find $H_3$ to be minimal inside $H\setminus(H_1\cup H_2)$ such that $H\setminus(H_1\cup H_2)\subset \cone(H_3)$. This process will end in a decomposition of $H=H_1\cup H_2\cup \dots \cup H_k$, such that $H_i$ contains two linearly independent elements of $H$ for all $i\in\{1,\ldots,k-1\}$, and $H_k$ may have one or two linearly independent elements of $H$. 

Given a finite subset $X$ of $M$, define 
\[
D(X)=\left\{ \sum_{x\in X} \alpha_x x \mid \alpha_x\in \mathbb{Q}, 0\le \alpha_x<1 \hbox{ for all }x\in X\right\}.
\]   

\begin{lemma}\label{lem:ap-diamond}
	For every $i\in \{1,\ldots,k-1\}$, $\Ap(M,H_1\cup\dots\cup H_i)\subset D(H_i)$.
\end{lemma}
\begin{proof}
	Let $H_i=\{(a,b),(c,d)\}$. Let $(x,y)\in \Ap(M,H_1\cup \dots \cup H_i)$. Hence $(x,y)\in \langle H_{i+1}\cup\dots\cup H_k\rangle$. Also, by construction, $\cone(H\setminus (H_1\cup \dots \cup H_i))=\cone(H_{i+1}\cup\dots\cup H_k) \subseteq \cone(H_i)$. Hence, $(x,y)=\alpha(a,b)+\beta(c,d)$ for some $\alpha,\beta\in \mathbb{Q}$ with $\alpha\ge0$ and $\beta\ge 0$. If $\alpha>1$, then $(x,y)-(a,b)\in \operatorname{G}(M)\cap \cone(M)\subseteq \operatorname{G}(M)\cap \mathbb{N}^2= M$, contradicting that $(x,y)\in\Ap(M,H_1\cup \dots \cup H_i)$. Thus $\alpha<1$ and a similar argument proves that $\beta<1$, yielding $(x,y)\in D(H_i)$.
\end{proof}	

\begin{lemma}\label{lem:unique-diamond}
	Let $u,v\in \mathbb{N}^2$ be linearly independent. Assume that $\alpha u+\beta v+w=w'$ for some $\alpha,\beta\in \mathbb{Z}$, and  $w,w'\in D(\{u,v\})$. Then $\alpha=\beta=0$ and $w=w'$.
\end{lemma}
\begin{proof}
	Write $u=(u_1,u_2)$ and $v=(v_1,v_2)$. As $w,w'\in D(\{u,v\})$, there exists $\gamma,\delta,\gamma',\delta'\in \mathbb{Q}\cap[0,1)$ such that $w=\gamma u+\delta v$ and $w'=\gamma' u+\delta' v$. Then $(\alpha+\gamma)u+(\beta+\delta)v=\gamma'u+\delta' v$. Since $u$ and $v$ are linearly independent, we obtain $0\le \alpha+\gamma=\gamma'<1$ and $0\le\beta+\delta=\delta'<1$, and this forces $\alpha=\beta=0$. 
\end{proof}

\begin{remark}
	The same result can be reached if we start with $\alpha u+w=\beta v+w'$, as we obtain $0\le \alpha+\gamma=\gamma'<1$ and $0\le\delta=\beta+\delta'<1$ to conclude the same goal.  
\end{remark}

\begin{theorem}\label{th:stratification-full-N2}
	Let $M$ be a simplicial full affine semigroup contained in $\mathbb{N}^2$ with $\operatorname{rank}(\operatorname{G}(M))=2$, and let $H$ be the Hilbert basis of $M$. Then there exists a partition of $H$, $H=H_1\cup \dots \cup H_k$, such that  for every $n\in M$, there exist unique $h_1,\ldots,h_k$, with $n=h_1+\dots+h_k$, 
	\begin{enumerate}
		\item $h_i\in \langle H_i\rangle$ for all $i$, and 
		\item for all $i\ge 2$, $h_i+\dots+h_k\in \Ap(M,H_1\cup \dots\cup H_{i-1})$.
	\end{enumerate}	
\end{theorem}
\begin{proof}
	Let $n\in M$. As $M$ is a simplicial affine semigroup with $H_1$ its set of extremal rays, $n=h_1+w_1$ with $h_1\in \langle H_1\rangle$, and $w_1 \in \Ap(M,H_1)$ \cite[Lemma~1.4]{c-m}. Now $w_1\in \Ap(M,H_1)\subset \langle H_2\cup \dots \cup H_k\rangle$, and so it can be written as $w_1=h_2+w_2$, with $h_2\in \langle H_2\rangle$ and $w_2\in \Ap(M,H_2)$ (by construction, the affine semigroup  $\langle  H_2\cup \dots \cup H_k \rangle$ is simplicial and $\cone(H_2)=\cone( H_2\cup \dots \cup H_k)$). As $w_1\in \Ap(M,H_1)$, we also have that $w_2\in \Ap(M,H_1)$; hence $w_2\in \Ap(M,H_1\cup H_2)$. This forces $w_2\in \langle H_3 \cup \dots \cup H_k\rangle$. Notice that after a finite number of steps we obtain $n=h_1+\dots+ h_k$ with the $h_i$ fulfilling the conditions in the statement. 
	
	Let us proof the uniqueness. Assume that $n=h_1+\dots+h_k=h_1'+\dots+h_k'$ with $h_i$ and $h_i'$ fulfilling the conditions (1) and (2) for all $i$. Write $w_1=h_2+\dots+h_k$ and $w_1'=h_2'+\dots+h_k'$. This means that $w_1,w_1'\in \Ap(M,H_1)$, and thus by Lemma~\ref{lem:ap-diamond}, $w_1,w_1'\in D(H_1)$. Lemma~\ref{lem:unique-diamond} forces $h_1=h_1'$ and $w_1=w_1'$. We repeat the same argument with the equality $h_2+\dots+h_k=h_2'+\dots+h_k'$.
\end{proof}

Observe that $\langle H_i\cup \dots \cup H_k\rangle$ for $i>1$ needs not to be full. In Example~\ref{ElliottExample}, $\langle (1,3), (5,1), (3,2)\rangle$ is not a full semigroup (indeed the ``full closure'' $\operatorname {G}(\langle (1,3), (5,1), (3,2)\rangle)\cap \mathbb{N}^2$  is  $M$).

Elliott's approach uses generating functions of the monoid $N$, and gives explicit solution to his problem of describing solutions in a unique way in terms of a ``simple set of solutions'' for $c$ up to 10. Our perspective (Theorem~\ref{th:stratification-full-N2}) deals directly with the atoms of the monoid $M$, and works for any $c$.

\section{Strong atoms, pure elements and extremal rays} \label{sec:extraction}

Let us see how we can generalize the notion of successively taking $H_i$ in the former section. We will see that the elements in $H_i$ are closely related to what it is known in the literature as strong atoms.

Let $M$ be a monoid, and $a,b\in M$. We write $a\le_M b$ if $b-a\in M$. Note that as we are assuming that $M$ is cancellative and reduced, the binary relation $\le_M$ is a partial order on $M$. 

An element $x\in M\setminus \{0\}$ is \emph{pure} if, for any $y\in M\setminus\{0\}$ with $y\leq_M kx$ for some $k\in\mathbb N\setminus\{0\}$, it follows that $mx=ny$ for some $m,n\in\mathbb N\setminus \{0\}$.

A \emph{strong atom} is a non-unit element $x$ of $M\setminus \{0\}$, such that, if $y\leq_M kx$ for some $y\in M$ and $k\in \mathbb N$, then there exists $l\in \mathbb N$ such that $y=lx$ \cite[Definition~3.1]{cale}. Notice that if $x$ is a strong atom of $M$, then $nx$ admits a unique expression in terms of atoms of $M$, and thus $x$ is absolutely irreducible (see  \cite[Definition~7.1.3]{g-hk}). Also, observe that the notion of strong atom coincides with the concept of completely fundamental element given in \cite{stanley}.

A \emph{face} of a monoid $M$ is a submonoid $N$ of $M$ such that whenever $a+b\in N$ for $a,b\in M$, one gets $a,b\in N$.


\begin{lemma}\label{lem:strong-face}
	Let $a$ be an atom of the monoid $M$. Then $a$ is strong if and only if $\mathbb{N}a=\{ma \mid m\in \mathbb N\}$ is a face of $M$.
\end{lemma}
\begin{proof}
	Assume that $a$ is strong, and take $x,y\in M$ such that $x+y= n a$ for some nonnegative integer $n$. Then $x\le_M n a$, and by definition, $x=k a$ for some nonnegative integer $k$. The same holds for $y$.
	
	For the converse, assume that $x\le_M k a$ for some $x\in M$ and some nonnegative integer $k$. Then there exists $y\in M$ such that $x+y=k a$, and as $\mathbb{N}a$ is a face, both $x$ and $y$ are in $\mathbb{N}a$. In particular, $x= l a$ for some $l\in \mathbb{N}$.
\end{proof}

Obviously, a strong atom is a pure atom. The following lemma gives a condition for the reverse implication. Two different elements $x,y\in M\setminus \{0\}$ are called \emph{disjoint} if $\mathbb Nx\cap\mathbb Ny=\{0\}$.

\begin{lemma}\label{purestrong} Let $M$ be atomic such that any two different atoms are disjoint. Then an atom is strong precisely if it is pure.
\end{lemma}
\begin{proof}
	Let $x$ be a pure atom and suppose $z\leq_M kx$. Take $y$ atom such that $y\leq_M z\leq_M kx$. Then $mx=ny$ with $m,n\in\mathbb N\setminus\{0\}$. Thus $\mathbb N x\cap\mathbb N y\neq\{0\}$ and by disjointness we must have $y=x$.
	
	This shows that $x$ is the only atom dividing $z$ which implies (as $M$ is atomic) that $z=lx$ for some $l\geq 1$. 
\end{proof}

Let $M$ be a monoid. The \emph{extraction grade} \cite[Definition page 149]{krause} for $x,y\in M\setminus\{0\}$ is
\[
\lambda_M (x,y)=\sup\{m/n \mid mx\leq_M ny, n,m\in\mathbb{N}, n\neq 0\}\in \mathbb R_+\cup\{\infty\}.
\]
The extraction grade $\lambda_M(x,y)$ measures in some sense how much of $x$ can be extracted from $y$.

The monoid $M$ is called an \emph{extraction monoid} if the extraction grade attains rational values, that is, for any $x,y\in M\setminus\{0\}$ there exist $m\in \mathbb Z_+$ and $n\in \mathbb N$ such that $mx\leq_M ny$ and $\lambda(x,y)=m/n$.

An extraction monoid $M$ is said to be of \emph{finite type} if for any $\{x_i\}_{i\in\mathbb{N}}\subseteq M$ the sequence $\mathrm{rad}(x_1)\subseteq\mathrm{rad}(x_2)\subseteq \ldots$ becomes stationary, where \[\mathrm{rad}(x)=\{y\in M\mid x\leq_M ny \mbox{ for some }n\in \mathbb N\}\] is the \emph{radical} of the principal ideal $(x)=x+M$.

\begin{theorem}\label{generatedbypure}
	Let $M$ be an extraction monoid of finite type.
	\begin{enumerate}
		\item[(i)] For each $x\in M\setminus\{0\}$ there exists $m\in \mathbb N\setminus \{0\}$ such that $mx$ is contained in a factorial submonoid of $M$ generated by pure elements of $M$.
		\item[(ii)] If in addition $M$ is atomic and any two different atoms are disjoint, then the factorial submonoid in (i) can be chosen to be generated by strong atoms of $M$. 
	\end{enumerate}
\end{theorem}
\begin{proof}\hspace{2mm}
	\begin{enumerate}
		\item[(i)] This is proven in \cite[Corollary page 151]{krause}. (The operation of the monoid there is multiplication instead of addition.)
		\item[(ii)] Follows from (i) with the help of Lemma~\ref{purestrong}. 
		\begin{itemize}
			\item[(a)] Since $M$ is atomic, for each pure element there is a pure atom below it, as the following argument shows. Let $x$ be a pure element of $M$, and $a$ be an atom of $M$ with $a\le_M x$. Then by definition of pure element, there exist positive integers $u$ and $v$ such that $ua=vx$. Take now $y\in M$ and $k$ a positive integer with $y\le_M ka$. Then $y\le_M kx$, and since $x$ is pure, there exist positive integers $m$ and $n$ with $mx=ny$. But then, $(mu)a=m(vx)= (nv)y$, which proves that $a$ is a pure atom.
			\item[(b)] Let $z\in M\setminus\{0\}$. By (i) there exist $m\in\mathbb N\setminus\{0\}$ and $x$ pure such that $x\leq_M mz$. By (a) there exists a pure atom $a\leq_M x$ and, hence, $a\leq_M mz$. From Lemma \ref{purestrong} we get that $a$ is a strong atom. Let $A$ be the set of all strong atoms in $M$. From \cite[Theorem 1]{krause} it follows that some (positive) multiple of every element in $M\setminus\{0\}$ is contained in a factorial submonoid generated by finitely many elements from $A$. \qedhere
		\end{itemize}
	\end{enumerate}
\end{proof}

An important consequence of Theorem \ref{generatedbypure} (ii) is that, under the assumptions made, the monoid possesses enough strong atoms. As we will see later, in affine semigroups, pure atoms correspond to extremal rays (and if the monoid is root-closed, these also coincide with strong atoms).

In general, however, a monoid need not possess strong atoms. For example, the numerical semigroup $\langle2,3\rangle$, which is an extraction monoid of finite type, has the two atoms $2$ and $3$ none of which is strong. By Theorem \ref{generatedbypure}, therefore, $2$ and $3$ cannot be disjoint; indeed the corresponding rays both contain $6$. Of course, by the first part of Theorem \ref{generatedbypure} there exist enough pure elements; indeed all nonzero elements are. The above example also shows that pure atoms need not be strong. 

Observe that in Theorem \ref{generatedbypure} for two different elements, some multiple for each of them belongs to a factorial submonoid, but these submonoids are different from each other in general. The special case where those submonoids can be chosen to be the same is of particular interest. A monoid is called \emph{inside factorial} if there is a factorial monoid $S\subseteq M$ such that for each $x\in M\setminus\{0\}$ there exists $m\in\mathbb N\setminus\{0\}$ such that $mx\in S$, (in other words, $M$ is a root-extension of a factorial monoid). The set of atoms $\mathcal A(S)$ is called the \emph{base} of the inside factorial monoid $M$. 

Although related in a simple way to factorial monoids, inside factorial monoids can be quite complicated as will be exhibited in Section 5. It is easy to see that any nonzero submonoid of $\mathbb N$ is inside factorial. Numerical semigroups are special cases and can be quite tricky (for more on numerical semigroups see \cite{RG}; a simple case is the above example). This indicates that a geometrical view of monoids as cones, though very often helpful, has to be taken very carefully. 

One might think that the condition that any two atoms are disjoint is too restrictive. The following result shows that there is a large family of monoids with this property.

We say that $M$ is \emph{root-closed} if whenever $na\le_M nb$ for some $a,b\in M$ and some positive integer $n$, then $a\le_M b$ (equivalently, if $n(b-a)\in M$ for some positive integer $n$ and some $a,b\in M$, then $b-a\in M$). 

\begin{lemma}\label{disjoint-atoms-krull}
	Let $M$ be a reduced root-closed monoid. Then any two different atoms are disjoint.
\end{lemma}
\begin{proof}
	Let $x$, $y$ be atoms such that $\mathbb N x\cap \mathbb N y\neq \{0\}$, that is $mx=ny$, for some $m,n\in\mathbb N\setminus\{0\}$. Without loss of generality, assume that $m\leq n$. Then $ny\leq_M nx$, and since $M$ is root-closed, we have $y\leq_M x$. Since $x$ and $y$ are atoms and $M$ is reduced, we get $y=x$. Therefore, any two different atoms must be disjoint.
\end{proof}

The converse of this result is not true. Take for instance the submonoid of $\mathbb{N}^2$ generated by $(2,0)$, $(1,1)$ and $(0,3)$. Any two atoms are disjoint. As $(0,1)=(0,3)-2(1,1)+(2,0)$, we have that $(0,1)\in \operatorname{G}(M)$, $3(0,1)\in M$ but $(0,1)\not\in M$. So $M$ is not root closed.

In light of Lemmas \ref{purestrong} and \ref{disjoint-atoms-krull}, strong and pure atoms coincide in root-closed atomic monoids. 

\begin{corollary}
	In a root-closed, atomic monoid, an atom is pure if and only if it is strong.	
\end{corollary}

Recall that a monoid is a \emph{Krull monoid} if $M=\{x\in \operatorname{G}(M)\mid f(x)\ge 0\hbox{ for all } f\in F\}$, for some set $F$ of nonzero group homomorphisms from $\operatorname{G}(M)$ to $\mathbb{Z}$, such that the set $\{f\in F\mid f(x)\neq 0\}$ is finite for every $x\in \operatorname{G}(M)$ (see \cite[Section~2.3]{g-hk} for alternative characterizations). In particular, Krull monoids are root-closed. In the finitely generated case with torsion free quotient group, root-closed and Krull monoids coincide (see \cite[Proposition~2]{KL} and \cite[Theorem~2.7.14]{g-hk}).  
Root-closed inside factorial monoids are rational generalized Krull monoids with torsion t-class group \cite[Theorem 3]{inside}.

\begin{corollary}\label{ForKrull}
	If $M$ is a Krull monoid, then for each $x\in M\setminus\{0\}$ there exists $m\in\mathbb N\setminus\{0\}$ such that $mx$ is contained in a factorial submonoid of $M$ generated by strong atoms of $M$.
\end{corollary}
\begin{proof}
	Being a Krull monoid, $M$ is atomic and root-closed. Furthermore, $M$ is an extraction monoid of finite type (see \cite[Section~4]{krause}). The corollary follows from part (ii) of Theorem~\ref{generatedbypure} and Lemma~\ref{disjoint-atoms-krull}.
\end{proof}

\begin{remark}\hspace{2mm}
	The representation of $mx$  in Theorem~\ref{generatedbypure} may be viewed as a version for monoids of Caratheodory's Theorem in convex analysis.
	
	The representation can be obtained by the extraction algorithm as on \cite[Algorithm page 150]{krause}.
	
	Corollary~\ref{ForKrull} has been recently proven by G. Angerm\"uller in \cite[Theorem~1 (a)]{augermuller}, in a slightly different language. His proof uses, beside extraction, the  divisor theory for Krull monoids. Observe that Theorem~\ref{generatedbypure} (ii), does not require $M$ to be root-closed. Actually, it might be applied to submonoids of Krull monoids which in general are not root-closed.
\end{remark}



Examples of extraction monoids of finite type that are not necessarily Krull monoids are given in the next section.

\subsection{Affine semigroups}\label{sec:affine}

We see in this section how the concepts of pure and strong atoms translate in the scope of affine semigroups, and how the extraction grade can be computed for these monoids.

Recall that an \emph{affine semigroup} is a finitely generated submonoid of $\mathbb{N}^d$ for some positive $d$. According to Grillet's Theorem this is equivalent to being isomorphic to a finitely generated cancellative, reduced and torsion-free monoid (\cite{grillet}, see also \cite[Theorem 3.11]{fg}).

For a subset $S$ of an affine semigroup $M\subseteq \mathbb N^n$ denote by \[\cone(S)=\left\{\sum_{i=1}^k r_is_i\mid r_i\in\mathbb Q_+, s_i\in S\right\}\] the cone in the $\mathbb Q$-vector space $\mathbb Q^n$ generated by $S$.

\begin{lemma}\label{characterization}
	Every affine semigroup is an extraction monoid of finite type.
\end{lemma}
\begin{proof}
	Suppose $M\subseteq \mathbb N^n$ is generated by the atoms $x_1,\ldots, x_d$. Consider $C=\cone(\{x_1,\ldots ,x_d\})=\{r_1x_1+\cdots +r_dx_d \mid r_i\in \mathbb Q_+\}$, the cone spanned by $\{x_1,\dots,x_d\}$. By the Farkas-Minkowski-Weyl Theorem (see for instance \cite[Section~7.2]{schrijver}) for vector spaces over an ordered field, $C$ is the intersection of finitely many half-spaces. That is,
	\[
	C=\left\{x\in\mathbb Q^d\mid v_i(x)\geq 0\mbox{ for }i\in\{1,\ldots, k\}\right\},
	\]
	where $v_i\colon \mathbb Q^d\rightarrow \mathbb Q$ is $\mathbb{Q}$--linear.
	
	We shall show that $$\lambda_M(x,y)=\min\{v_i(y)/v_i(x)\mid v_i(x)> 0, i\in\{1,\ldots ,k\}\}.$$
	
	The minimum on the right hand side is in $\mathbb Q_+$, whence equal to some $r/s$ with $r\in\mathbb N$, $s\in\mathbb N\setminus \{0\}$. Let $x,y\in M\setminus \{0\}$, $mx\leq_M ny$. Then, $ny=mx+z$ for some $z\in M$, and thus $n v_i(y)=m v_i(x)+ v_i(z)$, which implies $n v_i(y)\geq m v_i(x)$. Therefore $\frac{m}n\leq\frac{v_i(y)}{v_i(x)}$ for all $i$ such that $v_i(x)> 0$. This shows $\lambda_M(x,y)\leq r/s$. Conversely, $\frac{r}s\leq \frac{v_i(y)}{v_i(x)}$ for any $i$ with $v_i(x)>0$ (which is also true for any $i$ with $v_i(x)=0$ since $v_i(y)\geq 0$, because $y\in M\subseteq C$) and, hence, $sy-rx\in C$. By definition of $C$, $sy-rx=\sum_{i=1}^d r_ix_i$, with $r_i\in \mathbb Q_+$, for all $i$. There exists $k\in\mathbb N\setminus\{0\}$ such that $k(sy-rx)\in M$, that is $krx\leq_M ksy$ which implies $\lambda_M(x,y)\geq \frac{kr}{ks}=\frac{r}s$. Thus, we have $\lambda_M(x,y)=\frac{r}s$.
	
	To see that $M$ is of finite type, consider $I(x)=\{i\in\{1,\dots,k\} \mid v_i(x)>0\}$. From $\lambda_M(x,y)=\min\{v_i(y)/v_i(x)\mid v_i(x)>0\}$ we obtain that $\lambda_M(x,y)>0$ if and only if $I(x)\subseteq I(y)$. From the definition of $\operatorname{rad}(x)$, it follows that $\operatorname{rad}(x)\subseteq \operatorname{rad}(y)$ if and only if there exists a positive integer $m$ such that $y\le_M mx$, which in turn is equivalent to $\lambda_M(y,x)>0$. Putting all this together, we deduce that $\operatorname{rad}(x)\subseteq \operatorname{rad}(y)$ if and only if $I(x)\subseteq I(y)$. Since the sets $I(*)$ are finite, any sequence $\dots \subseteq I(x_2)\subseteq I(x_1)$ must become stationary, that is, $M$ is of finite type.	
\end{proof}

\begin{theorem}\label{propertiesofgenerated}
	Let $M$ be an affine semigroup and $A$ a finite set of atoms in $M$. The submonoid $S$ of $M$ generated by $A$ has the following properties.
	\begin{itemize}
		\item[(i)] For each $x\in S\setminus \{0\}$ there exists $m\in\mathbb N\setminus\{0\}$ such that $mx$ is contained in a factorial submonoid of $S$ generated by pure atoms of $S$.
		\item[(ii)] If any two different atoms in $A$ are disjoint, then the factorial submonoid in (i) can be chosen to be generated by strong atoms of $S$. 
	\end{itemize}
\end{theorem}
\begin{proof}
	Being finitely generated, $S$ is an extraction monoid of finite type by Lemma~\ref{characterization}. Theorem~\ref{generatedbypure}(i) implies property (i). Property (ii) follows from Theorem~\ref{generatedbypure}(ii) applied to $S$. 
\end{proof}	
\begin{remark} The representations in Theorem~\ref{propertiesofgenerated} can be obtained by the extraction algorithm. Part (i) applies also to numerical semigroups. In that case, however, there is no representation as in part (ii). Indeed, different atoms are never disjoint.
\end{remark}


Let $M\subseteq \mathbb{N}^n$ be an affine semigroup. Recall that we say that $M$ is \emph{full} (or saturated) if $\operatorname{G}(M)\cap \mathbb{N}^n=M$. Observe that, in such case, for any $x,y\in M$, $x\le y$ if and only if $x\le_M y$ (here $\le$ denotes the usual partial ordering on $\mathbb{N}^n$, that is, $a\le b$ if $b-a\in \mathbb{N}^n$). It is easily seen that ``full'' implies ``root-closed'': if $na\le_M nb$ for some $a,b\in M$ and a positive integer $n$, then $na\le nb$, and consequently $a\le b$, which implies $a\le_M b$.

The set $\mathcal{A}(M)$ is known in the literature as a \emph{Hilbert basis} of $M$. Notice that $M$ corresponds with the set of non-negative integer solutions of the defining equations of $\operatorname{G}(M)$.

Every full affine semigroup is isomorphic to a Krull monoid \cite[Proposition~2]{KL}.

We say that an atom $a$ of $M$ is a \emph{extremal ray} if $\mathbb Q_+ a$ is a face of $\cone(M)$.

\begin{remark}
	Let $a$ be an atom that is not an extremal ray, and denote by $A$ the set of atoms of the monoid $M$. Then, there exists $x,y\in \cone(M)$ such that $x+y\in \mathbb Q_+a$ and $x\notin \mathbb Q_+a$. Let $\lambda\in \mathbb{Q}_+$ be such that $x+y=\lambda a$. Assume that $y\in \mathbb Q_+a$. Then $y=\mu a$ for some $\mu\in\mathbb Q_+$, and $x+y=x+\mu a=\lambda a$. This implies that $x=(\lambda-\mu)a\in\mathbb Q_a$, and as $x\in\cone(M)$, we deduce that $x\in \mathbb{Q}_+a$, a contradiction. Hence $y\notin \mathbb Q_+a$. Write $x=\sum_{\bar a\in A} \lambda_{\bar a} \bar a$, $y=\sum_{\bar a\in A} \mu_{\bar a} \bar a$, $\lambda_{\bar a}, \mu_{\bar a}\in \mathbb Q_+$. If $\lambda_a>0$, $(x-\lambda_aa)+y=(\lambda-\lambda_a)a\in \mathbb Q_+a$ and $x-\lambda_aa\notin \mathbb Q_+a$. So we may assume that $\lambda_a=0=\mu_a$. This implies that $a\in\sum_{\bar a\in A\setminus\{a\}}\mathbb Q_+\bar a$, and $\cone(M)=\cone(A)=\cone(A\setminus \{a\})$. This shows that the cone spanned by $M$ equals the cone spanned by its extremal rays.
\end{remark}

\begin{lemma}\label{lem:pure-face}
	Let $M$ be an affine semigroup and let $x$ be an atom of $M$. Then $x$ is pure if and only if it is an extremal ray.
\end{lemma} 
\begin{proof}
	Assume that $x$ is a pure atom. Take $y,z\in \cone(M)\setminus\{0\}$ with $y+z \in \mathbb{Q}_+ x$. Then there exist positive integers $k$, $r$ and $s$ such that $ky, kz\in M$ and $y+z=\frac{r}s x$. Then $sky+skz =kr x$, and consequently $sky\le_M kr x$. As $x$ is pure, there exist positive integers $n,m$ such that $nsky =mkrx$, and thus $y\in \mathbb{Q_+}x$. A similar argument shows $z\in \mathbb{Q_+}x$.
	
	Now assume that $x$ is an extremal ray of $M$, and assume that $y\le_M kx$ for some $y\in M\setminus\{0\}$ and some positive integer $k$. Then there exists $z\in M$ with $y+z=kx$. In particular, $y+z\in \mathbb{Q_+}x$, and as $\mathbb{Q_+}x$ is a face of $\cone(M)$, we deduce that $y\in \mathbb{Q_+}x$. Thus, there exist positive integers $n$ and $m$ with $y=\frac{m}n x$, and so $ny=mx$.
\end{proof}

In light of Lemmas  \ref{purestrong} and \ref{disjoint-atoms-krull}, strong and pure atoms coincide in full affine semigroups; this combined with Lemmas \ref{lem:strong-face} and \ref{lem:pure-face} yields the following equivalences.

\begin{corollary}
	Let $M$ be a full affine semigroup and let $a$ be an atom of $M$. The following are equivalent:
	\begin{itemize}
		\item $a$ is a strong atom;
		\item $a$ is a pure atom;
		\item $a$ is an extremal ray;
		\item $\mathbb{N}a$ is a face of $M$.
	\end{itemize}
\end{corollary}

Notice that as a consequence of Lemma \ref{disjoint-atoms-krull}, if $a$ is an extreme ray of a full affine monoid $M$, then the only atom of $M$ in $\mathbb{Q}_+a$ is $a$. 

For full affine semigroups, Ap\'ery sets can be computed with the help of extraction degrees. 

\begin{proposition}\label{prop:apery-lambda}
	Let $M$ be a full affine semigroup and let $A$ be a subset of $M$. Then 
	\[
	\Ap(M,A) = \{ y \in M \mid \lambda_M(x,y)<1 \hbox{ for all } x \in A\}.
	\] 
\end{proposition}
\begin{proof}
	Notice that as $M$ is full, $x\le y$ if and only if $x\le_M y$ for every $x,y\in M$. Thus $\lambda_M(x,y)=\sup\{ m/n\mid m x \le n y, n,m\in \mathbb{N}, n\neq 0\}$. Observe that if $\lambda_M(x,y)=m/n \ge 1$, then $m x\le ny $ leads to $x\le \frac{m}n x \le y$, and consequently $x\le y$. Also if $x\le y$, then trivially $\lambda_M(x,y)\ge 1$. This proves that $\lambda_M(x,y)<1$ if and only if $x\not\le_M y$, or equivalently $y\not\in x+M$. Thus $\lambda_M(x,y)<1$ if and only if $y\in M\setminus(x+M)$. With this observation, the proof follows easily.
\end{proof}

With this, we can reformulate Example~\ref{ElliottExample}, and parametrize the set of nonnegative integer solutions of $x+2y\equiv 0 \pmod 7$. The proof of Lemma~\ref{characterization} gives a way to compute $\lambda_M((x_1,y_1),(x_2,y_2))$, for $(x_1,x_2)$ and $(y_1,y_2)$ elements of $M$, with $(x_1,x_2)$ nonzero. Observe that the facets of the cone spanned by $M$ are in the coordinate axes \cite[Section~2]{rgs-full}, so we can take $v_1(x_1,x_2)=x_1$ and $v_2(x_1,x_2)=x_2$. In particular, 
\begin{equation}\label{formula-lambda}
	\lambda_M((x_1,x_2),(y_1,y_2))=\begin{cases}
		\min\left\{\dfrac{y_1}{x_1},\dfrac{y_2}{x_2}\right\}, &\hbox{if  } x_1\neq 0 \neq x_2, \\[10pt]
		\dfrac{y_1}{x_1}, &\hbox{if } x_2=0,\\[10pt]
		\dfrac{y_2}{x_2}, &\hbox{if } x_1=0.
\end{cases}\end{equation}

\begin{example}\label{ex:Elliot-lambda}
	Let  $M$ be the set of non-negative integer solutions of $x+2y\equiv 0 \pmod 7$. We already know that every $n\in M$ will be expressed as $n=h_1+h_2+h_3$, with 
	\begin{enumerate}[(a)]
		\item $h_1\in \langle (7,0),(0,7)\rangle$,
		\item $h_2\in\langle (1,3),(5,1)\rangle$, say $h_2=\alpha (1,3)+\beta(5,1)$,
		\item $h_3\in \langle (3,2)\rangle$, say $h_3=\gamma (3,2)$,
	\end{enumerate}
	and $h_2+h_3\in \Ap(M,H_1)$, and $h_3\in \Ap(M,H_1\cup H_2)$. Let us use Proposition~\ref{prop:apery-lambda}. The restrictions \[\lambda_M((1,3),h_3)<1, \ \lambda_M((5,1),h_3)<1,\] and \eqref{formula-lambda} force $\gamma\in\{0,1\}$; while the restrictions \[\lambda_M((7,0),h_2+h_3)<1,\ \lambda_M((0,7),h_2+h_3)<1\] translate to \[\alpha+5\beta+3\gamma<7,\  3\alpha+\beta+2\gamma<7,\] respectively. This limits $\beta$ to be at most one also, and $\alpha$ to be at most two. 
	
	So every element $n\in M$ is written uniquely as 
	\[
	n=\delta(7,0)+\eta(0,7)+\alpha(1,3)+\beta(5,1)+\gamma(3,2),
	\]
	with $\delta,\eta,\alpha,\beta,\gamma\in \mathbb{N}$, subject to 
	\[\alpha+5\beta+3\gamma<7,\  3\alpha+\beta+2\gamma<7,\ \gamma<2.\]
	Compare with \cite[Case~XXVI]{elliot}.
\end{example}

\subsection{Elementary atoms of block monoids}

We now relate the concept of strong atom to elementary zero-sum sequences, which were introduced in \cite{b-g-g-s}.

Let $(G,+)$ be a group and let $G_0\subseteq G$. For a sequence $x=(x_g)_{g\in G_0}$, we define its support as $\supp(x)= \{ g\in G_0 \mid x_g\neq 0\}$. We say that $x$ is a \emph{zero-sum sequence} if its \emph{support} is finite and $\sum_{g\in G_0} x_g g=0$ (in $G$). If we have two zero-sum sequences $(x_g)_{g\in G_0}$ and $(y_g)_{g\in G_0}$, then $(x_g+y_g)_{g\in G_0}$ is again a zero sequence. So if we set $\mathcal{B}(G_0)$ to be the set of zero sum sequences over $G_0$, then $\mathcal{B}(G_0)$ is a monoid, known as the \emph{block monoid} over $G_0$. Block monoids are atomic Krull monoids \cite[Proposition~2.5.6]{g-hk}. 

Given two zero-sum sequences $(x_g)_{g\in G_0}$ and $(y_g)_{g\in G_0}$, we write $(x_g)_{g\in G_0}\le (y_g)_{g\in G_0}$ if $x_g\le y_g$ for all $g\in G_0$. Notice that if $(x_g)_{g\in G_0}\le (y_g)_{g\in G_0}$, then $(x_g-y_g)_{g\in G_0}$ is again a zero-sum sequence. So in some sense, block monoids can be seen as a generalization of full affine semigroups. Let $M\subseteq \mathbb{N}^n$ be a full affine semigroup. Then $\operatorname{G}(M)$ is a subgroup of $\mathbb{Z}^n$. Let 
\[
\begin{array}{lcl}
  a_{11}x_1+\cdots+a_{1n}x_n & \equiv & 0 \pmod {d_1},\\
  & \vdots & \\
  a_{r1}x_1+\cdots+a_{rn}x_n & \equiv & 0 \pmod {d_r},\\
  a_{(r+1)1}x_1+\cdots+a_{(r+1)n}x_n & = & 0,\\
  & \vdots & \\
  a_{(r+k)1}x_1+\cdots+a_{(r+k)n}x_n & = & 0,\\
\end{array}
\]
be a set of defining equations of $\operatorname{G}(M)$. Let $G_0\subseteq \mathbb{Z}_{d_1}\times \dots \times \mathbb{Z}_{d_r}\times \mathbb{Z}^k$ be the set of columns of the above system of equations. It follows easily that $M$ is precisely $\mathcal{B}(G_0)$. 

A zero-sum sequence $x$ is \emph{elementary} if its support is not empty and minimal, that is, there is no other zero-sum sequence $y$ with $\emptyset\neq \supp(y)\subsetneq \supp(x)$. 

\begin{proposition}
Let $M$ be a block monoid and let $a$ be an atom of $M$. Then $a$ is elementary if and only if $a$ is strong.
\end{proposition}
\begin{proof}
Suppose that $a$ is a strong atom. Let $b$ be with nonempty support be such that $\supp(b)\subseteq \supp (a)$. Take $k$ a positive integer such that $b\le k a$, and so $b\le_M ka$. As $a$ is a strong atom, it follows that there exists a nonnegative integer $l$ such that $b=l a$, and thus $\supp (b) = \supp (a)$. 

Now, suppose that $a$ is an elementary atom, and let us prove that $a$ is strong. By Corollary~\ref{ForKrull}, there exists a positive integer $n$ such that $na = \sum_{q\in Q} \lambda_q q$ with $Q$ a set of strong atoms of $M$ and $\lambda_q$ nonnegative integers for all $q\in Q$. This in particular implies that $\supp (q)\subseteq \supp (a)$ for all $q\in Q$, and as $a$ is elementary, we have that $\supp (q)=\supp (a)$. Take $q\in Q$. We can find a positive integer $k$ such that $a\le k q$, and consequently $a\le_M k q$. As $q$ is a strong atom, it follows that there exists a nonnegative integer $l$ such that $a=l q$. Since $a$ is already an atom, $l$ must be one, and $a$ is a strong atom.
\end{proof}

\section{Stratification}

We say that an affine semigroup $M\subseteq \mathbb{N}^d$ is \emph{simplicial} if the cone $\cone(M)$ is spanned by $d$ $\mathbb{Q}$-linearly independent atoms of $M$, say $a_1,\ldots,a_d$. Observe that Theorem~\ref{propertiesofgenerated}~(ii) applies to each simplicial affine semigroup $M$ having a single atom in each one dimensional face of the cone spanned by $M$. By its definition, $\{a_1,\ldots,a_d\}$ generates a factorial submonoid of $M$ that contains for every $x\in M\setminus\{0\}$ some multiple $mx$, $m\in \mathbb N\setminus\{0\}$. Thus, simplicial affine semigroups are particular examples of inside factorial monoids. Whereas the latter are root-extensions of a factorial monoid, the former may be seen as $\mathbb Q_+$-extensions of a simplex. 

In order to find a generalization of Theorem~\ref{th:stratification-full-N2} for inside factorial monoids, we first need to find an analogue for \cite[Lemma~1.4]{c-m} and \cite[Theorem~1.5]{c-m}.

Let us start with \cite[Lemma~1.4]{c-m}. Suppose that we have a certain set of atoms $Q$ of a monoid $M$. If $x\in M$, and $x\not\in \Ap(M,Q)$, then $x=q_1+x_1$, with $q_1\in Q$ and $x_1\in M$. Now we start anew with $x_1\in M$, and if $x_1\not\in \Ap(M,Q)$, then there exists $q_2\in Q$ and $x_2\in M$ such that $x_1=q_2+x_2$, and so $x=q_1+q_2+x_2$. In this way we obtain a sequence $q_1+\dots+q_n$ of elements in $\langle Q\rangle$, such that $q_1+\dots+q_n\le_M x$. This motivates the following definition: we say that a non-empty subset $Q$ of a monoid $M$ is \emph{linked} to $M$ if for each $x\in M\setminus\{0\}$ there exists $c(x)\in \mathbb N$ such that whenever $\sum_{q\in Q}\lambda_qq\leq_M x$, for some $\lambda_q\in\mathbb N$, it follows that $\sum_{q\in Q}\lambda_q\leq c(x)$.

\begin{proposition}\label{prop:decomposition-apery}
	Let $M$ be a monoid, and let $Q$ be a nonempty set of nonzero elements of $M$. If $Q$ es linked to $M$, then 
	\[
	M= \Ap(M,Q) +\langle Q \rangle.
	\]
\end{proposition}
\begin{proof}
	Take $x\in M$. If $x\in\Ap(M,Q)$, then we can write $x=0+x$ and we have the desired decomposition. Otherwise, if $x\notin \Ap(M,Q)=M\setminus (M+Q)$, we can write $x=m+q$, with $m\in M$ and $q\in Q$, whence $0\neq q\leq _M x$. We consider $N=\{\sum\lambda_qq\in \langle Q\rangle \mid \sum\lambda_qq\leq_M x\}$ and take $h=\sum\lambda_qq\in N$ such that $\sum\lambda_q$ is maximal (this expression exists by the definition of linked). This means that $x=z+h$ for some $z\in M$, and, it is clear that, by maximality, $z\in\Ap(M,Q)$.
\end{proof}

As we see next, bases of inside factorial monoids are always linked to the monoid.

\begin{corollary}\label{cor:decoposition-apery-inside}
	Let $M$ be an inside factorial monoid with base $Q$. Then $Q$ is linked to $M$. In particular, $M= \Ap(M,Q) +\langle Q \rangle$.
\end{corollary}
\begin{proof}
	As $M$ is inside factorial with base $Q$, for every $x\in M$, there exist non negative integers $n_x$ and $x_q$ such that $n_xx=\sum_{q\in Q} x_qq$ (with $x_q$ nonnegative integers and all except finitely many of them equal to zero). We assert that $c_x$ can be chosen as $\sum_{q\in Q}\left \lfloor \frac{x_q}{n_x} \right\rfloor$. Suppose that $\lambda_q$ are nonnegative integers such that $\sum_{q\in Q} \lambda_qq$ is a finite sum and $\sum_{q\in Q} \lambda_qq \leq_M x$. 
	Then $m=x-\sum_{q\in Q} \lambda_q q\in M$, and consequently $n_m(x-\sum_{q\in Q}\lambda_qq)\in\langle Q\rangle$, for some positive integer $n_m$. Thus,  $n_xn_m(x-\sum_{q\in Q}\lambda_qq)\in\langle Q \rangle$. Therefore, $n_m\sum_{q\in Q}x_qq-\sum_{q\in Q}n_xn_m\lambda_qq\in\langle Q\rangle$. Since $\langle Q\rangle$ is factorial, we have $n_mn_x\lambda_q\leq n_mx_q$ for all $q\in Q$. This yields $\lambda_q\leq\frac{x_q}{n_x}$ and, hence, $\sum_{q\in Q}\lambda_q\leq \sum_{q\in Q}\left \lfloor \frac{x_q}{n_x}\right\rfloor$.
	By Proposition~\ref{prop:decomposition-apery}, we deduce that $M= \Ap(M,Q) +\langle Q \rangle$.
\end{proof}

Now, let us focus on \cite[Theorem~1.5]{c-m}. Given a monoid $M$ and two subsets $X$ and $Y$ of $M$, we denote as usual $X+Y=\{x+y \mid x\in X, y\in Y\}$. We write $X\oplus Y$ whenever for every $m\in X+Y$ there exist unique $x\in X$ and $y\in Y$ such that $m=x+y$.

Given $x,y\in M$, we say that $x$ is \emph{coprime} with $y$ if $x\not\le_M y$ and whenever $x \le_M m+y$ for some $m\in M$, we have $x\le_M m$. 

\begin{theorem}\label{th:decomposition-apery-coprime}
	Let $M$ be a monoid, and let $Q$ be a set of nonzero elements of $M$ such that $\langle Q\rangle$ is factorial and $Q$ is linked to $M$. Then, the following conditions are equivalent.
	\begin{itemize}
		\item[(1)] $M=\Ap(M,Q)\oplus \langle Q\rangle$.
		\item[(2)] For every $q,q'\in Q$, with $q\neq q'$, we have that $q$ is coprime with $q'$.
		\item[(3)] For every $w,w'\in \Ap(M,Q)$, if $w-w'\in \operatorname{G}(Q)$, then $w=w'$.
	\end{itemize}
\end{theorem}
\begin{proof}
	By Proposition~\ref{prop:decomposition-apery}, we know that $M=\Ap(M,Q)+\langle Q\rangle$.

	\emph{(1) implies (2)}. Suppose that $q'\leq_M x+q$, with $q,q'\in Q$, $q\neq q'$ and $x\in M$. Then there exists $x'\in M$ such that $x'+q'=x+q$. We can write $x=z+h$ and $x'=z'+h'$, for some $z,z'\in \Ap(M,Q)$ and $h,h'\in\langle Q\rangle$. So, $z'+(h'+q')=z+(h+q)$, and by the uniqueness hypothesis, we derive $z=z'$ and $h'+q'=h+q$. As $q\neq q'$ and $\langle Q\rangle$ is factorial, $q'\leq_M h\leq_M x$.
	
	\emph{(2) implies (3)}. Let $w,w'\in \Ap(M,Q)$ with $w-w'\in \operatorname{G}(Q)$. Then $w-w'=h-h'$ for some $h,h'\in\langle Q\rangle$. It follows that $w+h=w'+h'$. As $h, h'\in \langle Q\rangle$, there exists nonnegative integers $\lambda_q$ and $\lambda_q'$ such that $h=\sum_{q\in Q}\lambda_q q$ and $h'=\sum_{q\in Q} \lambda_q' q$. If  $\lambda_q\neq \lambda_q'$ for some $q\in Q$, we may suppose without loss of generality that $\lambda_q>\lambda_q'$. Then $q\le_M  w'+h'-\lambda_q'q$. As $q$ is coprime with all $q'\in Q\setminus\{q\}$, this clearly leads to $q\le_M w'$, contradicting the fact that $w'\in \Ap(M,Q)$. This means that $h=h'$, and as a consequence, $w=w'$. 
	
	\emph{(3) implies (1)}. Let $x\in M$ be such that $x=w+h=w'+h'$ with $w,w'\in \Ap(M,Q)$ and $h,h'\in \langle Q\rangle$. Then $w-w'\in \operatorname{G}(Q)$, and by hypothesis $w=w'$, which leads to $h=h'$. 
\end{proof}

With this we can recover \cite[Theorem~3]{GSKL}.

\begin{corollary}
	Let $M$ be a root-closed inside factorial with base $Q$. Then $M= \Ap(M,Q) \oplus \langle Q \rangle$, that is, for every $x\in M$, one has $x=\sum_{q\in Q}\lambda_qq+z$, where $\lambda_q\in\mathbb N$ and $z\in\Ap(M,Q)$ are uniquely determined.
\end{corollary}
\begin{proof}
		Let $q_1\leq_M x+q_2$ with $q_1,q_2\in Q$, $q_1\neq q_2$. Hence, $x+q_2=x'+q_1$ with $x'\in M$. As $M$ is inside factorial with base $Q$, we can find a nonnegative integer $n$, and nonnegative integers $\lambda_q$ and $\lambda_q'$ for all $q\in Q$, such that  $nx=\sum_{q\in Q} \lambda_q q$ and $nx'=\sum_{q\in Q} \lambda_q'q$. Hence  $\sum_{q\in Q} \lambda_q q+nq_2=\sum_{q\in Q} \lambda_q' q+nq_1$. Since $\langle Q\rangle$ is factorial and $q_2\neq q_1$ it follows $\lambda_{q_1}=\lambda_{q_1}'+n\geq n$. Therefore, $n(x-q_1)=(\lambda_{q_1}-n)q_1+\sum_{q\in Q,q\neq q_1} \lambda_q q\in M$. Since $M$ is root-closed, we get $x-q_1\in M$, that is $q_1\leq_M x$.
\end{proof}

Let $M$ be an atomic monoid with set of atoms $H$. A \textit{stratification} of $H$ is a partition $H=H_1\cup H_2\cup\dots\cup H_k$, $k\ge 1$, such that for all $i\in\{1,\ldots,k\}$, 
\begin{itemize}
	\item[(S1)] the monoid $\langle H_i\rangle$ is factorial, 
	\item[(S2)] the set $H_i$ is linked to the monoid $M_i=\langle H_i \cup \dots \cup H_k\rangle$.
\end{itemize}

\begin{theorem}\label{uniquerep2}
	Let $M$ be an atomic monoid whose set of atoms admits a stratification $H=H_1\cup H_2\cup\dots\cup H_k$.  Then each $x\in M$ has a representation of the form $x=h_1+h_2+\cdots +h_k$ such that
	\begin{enumerate}
		\item $h_i\in\langle H_i\rangle$ for all $i\in \{1,\dots, k\}$,
		\item $h_i+\cdots +h_k\in \Ap(M,H_1\cup\dots\cup H_{i-1})$ for all $i\in\{2,\ldots k\}$.
	\end{enumerate}
	If in addition for each $i\in\{1,\dots,k\}$, the condition 
	\begin{itemize}
		\item[(S3)] for every $q,q'\in H_i$ with $q\neq q'$, $q$ and $q'$ are coprime in $M_i$
	\end{itemize}
	holds, then the above representation is unique.
\end{theorem}
\begin{proof}
	Let $x\in M$. We prove by induction that for all $i\in\{1,\ldots,k\}$, $x=h_1+\dots +h_i+w_i$, with $h_i\in \langle H_i\rangle$ and $w_i\in \Ap(M,H_1\cup \dots\cup H_i)$; in particular $w_i\in M_{i+1}$ (for $i<k$). By definition, $H_1$ is linked to $M=\langle H_1\cup \dots \cup H_k\rangle$, and so Proposition~\ref{prop:decomposition-apery}, ensures that there exists $h_1\in \langle H_1\rangle$ and $w_1\in\Ap(M,H_1)$ such that $x=h_1+w_1$. From $w_1\in \Ap(M,H_1)$, we deduce that $w_1\in M_2=\langle H_2\cup \dots \cup H_k\rangle$. 
	For the step from $i$ to $i+1$, we use again Proposition~\ref{prop:decomposition-apery} with $H_{i+1}$ and $M_{i+1}$, and obtain $w_i=h_{i+1}+w_{i+1}$, with $h_{i+1}\in \langle H_{i+1}\rangle$ and $w_{i+1}\in \Ap(M_{i+1},H_{i+1})$. As $w_i=h_{i+1}+w_{i+1}$ and by induction hypothesis $w_i\in \Ap(M,H_1\cup\dots\cup H_i)$, we have that $w_{i+1}\in \Ap(M,H_1\cup \dots \cup H_i)$. Assume that $w_{i+1}-q\in M$ for some $q\in H_{i+1}$. Then $w_{i+1}-q\in M\setminus M_{i+1}$ (because $w_{i+1}\in \Ap(M_{i+1},H_{i+1})$). Every element in $M\setminus M_{i+1}$ is of the form $h+m$ with $0\neq h\in \langle H_1\cup \dots \cup H_i\rangle$ and $m\in M_{i+1}$. So $w_{i+1}=q+h+m$, which contradicts $w_{i+1}\in \Ap(M,H_1\cup \dots \cup H_i)$. This proves that $w_{i+1}\in  \Ap(M,H_1\cup \dots \cup H_{i+1})$. 
	For $i=k$, we get $w_k\in \Ap(M,H_1\cup \dots \cup H_k)=\{0\}$, and we have that $x=h_1+\dots+h_k$ fulfilling conditions (1) and (2).
	
	For the uniqueness, assume that $h_1+\dots+h_k=h_1'+\dots+h_k'$ with $h_i$ and $h_i'$ as in the statement. We show by induction over $i$ that $h_i+\dots+h_k= h_i'+\dots+ h_k'$ implies that $h_{i+1}+\dots+h_k= h_{i+1}'+\dots+ h_k'$. For $i=1$, Theorem~\ref{th:decomposition-apery-coprime} applied to $M=M_1$ and $Q=H_1$, asserts that $h_1=h_1'$ and $h_2+\dots+h_k=h_2'+\dots+h_k'$. Suppose now that $h_i+\dots+h_k=h_i'+\dots+h_k'$. From (1) and (2), $h_i,h_i'\in \langle H_i\rangle$ and $h_{i+1}+\dots+h_k, h_{i+1}'+\dots+h_k' \in \Ap(M,H_1\cup \dots \cup H_i)$. Notice that by Theorem~\ref{th:decomposition-apery-coprime}, $M_i=\Ap(M_i,H_i)\oplus \langle H_i\rangle$. Suppose that $h_{i+1}+\dots+h_k-q\in M_i$ for some $q\in H_i$. Then $h_{i+1}+\dots+h_k-q\in M$, contradicting that $h_{i+1}+\dots+h_k \in \Ap(M,H_1\cup \dots \cup H_i)$. Hence $h_{i+1}+\dots+h_k\in \Ap(M_i,H_i)$. The same holds for $h_{i+1}'+\dots+h_k'$, and so the decomposition of Theorem~\ref{th:decomposition-apery-coprime} implies that $h_i=h_i'$ and $h_{i+1}+\dots+h_k=h_{i+1}'+\dots+h_k'$.
%
\end{proof}

\begin{example}
	Let $S$ be a numerical semigroup minimally generated by $\{n_1,\ldots,n_e\}$. Then $H=H_1\cup \dots \cup H_e$, with $H_i=\{n_i\}$ is a stratification of the minimal generating system of $S$. Condition~(S3) holds trivially in this setting.
\end{example}

\begin{example}
	Let $M$ be a full affine semigroup of $\mathbb{N}^2$ with $\operatorname{rank}(\operatorname{G}(M))=2$. Let $H$ be a Hilbert basis of $M$. Then the partition provided in Theorem~\ref{th:stratification-full-N2} is a stratification of $H$. 
\end{example}

Let $M$ be an inside factorial monoid with basis $Q$. Let $x\in M$. Then there exists a positive integer $n$ and nonnegative integers $\lambda_q$ for every $q\in Q$ (all but a finite number of them equal to zero) such that $nx = \sum_{q\in Q}\lambda_q q$. Observe that if $n'x=\sum_{q\in Q} \lambda_q' q$, then $nn'x=\sum_{q\in Q} n'\lambda_q q= \sum_{q\in Q}n\lambda_q' q$, and as $\langle Q\rangle$ is factorial, $n'\lambda_q = n\lambda_q'$, for all $q\in Q$. Equivalently $\lambda_q/n=\lambda_q'/n'$ for all $q\in Q$. We will call the set $\{\lambda_q/n\}_{q\in Q}$ the set of \emph{coordinates} of $x$ with respect to $Q$ in $M$. So the coordinates of $x$ uniquely determine $x$. 

Let us recall the following properties of the extraction grade  \cite[Lemma~2]{inside}: 
	\begin{enumerate}
	\item $\lambda_M(q,ax+by)=a\lambda_M(q,x)+b\lambda_M(q,y)$ for any $q\in Q$, $x,y\in M$, $a,b\in \mathbb N$;
	\item $\lambda_M(q,q)=1$ and $\lambda_M(q,q')=0$ for any $q,q'\in Q$, $q\neq q'$.
\end{enumerate}
Then it follows easily that $\{\lambda_M(q,x)\}_{q\in Q}$ are precisely the coordinates of $x$ with respect to $Q$ in $M$.

We define $D_M(Q)$ as the set of elements in $M$ whose coordinates are upper bounded by $1$. We then have a similar result to Lemma~\ref{lem:unique-diamond}.

\begin{lemma}\label{lem:unique-diamond-inside}
	Let $M$ be an inside factorial monoid with base $Q$. Assume that $g+w=w'$ for some $g\in \operatorname{G}(Q)$ and some $w,w'\in D_M(Q)$. Then $g=0$ (and thus $w=w'$).
\end{lemma}
\begin{proof}
	Let $z_q\in \mathbb{Z}$ be such that $g=\sum_{q\in Q} z_q q$, and let $n$ be a positive integer such that $nw$ and $nw'$ are in $\langle Q\rangle$. So write $nw=\sum_{q\in Q} \lambda_q q$ and $nw'=\sum_{q\in Q} \lambda_q' q$, for some nonnegative integers $\lambda_q$, $\lambda_q'$. Then $ng+nw=nw'$, which implies that $\sum_{q\in Q} (nz_q+\lambda_q)q=\sum_{q\in Q} \lambda_q'q$. As \(\langle Q\rangle\) is factorial, we deduce that for all $q\in Q$, $nz_q+\lambda_q=\lambda_q'$, and consequently $0\le z_q+\lambda_q/n=\lambda_q'/n$. By hypothesis, $\lambda_q/n$ and $\lambda_q'/n$ are smaller than one, and $z_q\in \mathbb{Z}$, which forces $z_q=0$ for all $q\in Q$. 
\end{proof}

\begin{remark}
	For a cancellative monoid $M$, we can define its \emph{root-closure} as \[\bar{M}=\{x\in \operatorname{G}(M) \mid nx\in M \mbox{ for some } n\in \mathbb N\}.\] 
	With this notation, $M$ is root-closed if and only if $M=\bar M$. Suppose that $M$ is inside factorial with base $Q$. Then $\bar{M}$ is also inside factorial with base $Q$. Take $x\in M$ and let $\{r_q\}_{q\in Q}$ be the coordinates of $x$ with respect to $Q$ in $M$. Set $\bar{x}= x- \sum_{q\in Q} \lfloor r_q\rfloor q$. Then $\bar{x}\in \operatorname{G}(M)$. Let $n$ be a positive integer such that $nr_q\in \mathbb{N}$ for all $q$. Then $n\bar{x}\in M$, and consequently $\bar{x}\in D_{\bar{M}}(Q)$. This proves that $x=h+\bar{x}$, for some $h\in \langle Q\rangle$ and $\bar{x}\in D_{\bar{M}}(Q)$. Lemma~\ref{lem:unique-diamond-inside} applied to $\bar{M}$ shows that this expression is unique. 
\end{remark}

Observe that every numerical semigroup is inside factorial with base a singleton containing any nonzero element of the semigroup. In particular, we may choose a base containing just one atoms. We show next that for any inside factorial monoid, we can choose a base having only atoms.

\begin{lemma}
	Let $M$ be an inside factorial monoid. It is possible to choose a basis for $M$ whose elements are all atoms.
\end{lemma}
\begin{proof}
 Let $A$ be the set of atoms of $M$. Suppose that there exists $q_0\in Q\setminus A$. Hence, there is some $a\in A$ with $a+m=q_0$ for some $m\in M\setminus\{0\}$. From the definition of inside factorial monoid, $n_aa=\sum_{q\in Q} a_qq$ for some $n_a\in\mathbb N\setminus \{0\}$ and $a_q$ non negative integers, and $n_mm=\sum_{q\in Q} m_qq$ for some $n_m\in\mathbb N\setminus \{0\}$ and $m_q$ non negative integers. As $n_an_m(a+m)=n_an_mq_0$ we have $\sum_{q\in Q}(n_ma_q+n_am_q)q=n_an_mq_0$. Since $\langle Q \rangle$ is factorial, $n_m a_q+n_a m_q=0$ for $q\neq q_0$ and $n_ma_{q_0}+n_a m_{q_0}=n_an_mq_0$. From the first equality $a_q=0$ and $m_q=0$ for $q\neq q_0$, because all of them are nonnegative integers. So we obtain $n_aa=a_{q_0}q_0$. 

 Let $x\in M$. As $M$ is inside factorial with base $Q$, we can write $n_xx=\sum_{q\in Q} \lambda_qq$ for some positive integer $n_x$ and some $\lambda_q$ non negative integers. Hence 
 \[a_{q_0}n_xx=\sum_{q\in Q} a_{q_0}\lambda_q q=\sum_{q\in Q, q\neq q_0} a_{q_0}\lambda_q q+a_{q_0}\lambda_{q_0}q_0=\sum_{q\in Q, q\neq q_0} a_{q_0}\lambda_q q+\lambda_{q_0}n_a a.\] Notice that from the equality $n_aa=a_{q_0}q_0$ it easily follows that $\langle Q'\rangle$ is factorial, which shows that $Q'=Q\setminus \{q_0\}\cup\{a\}$ is a base for $M$.
\end{proof}

Now suppose that $M$ is an inside factorial monoid with base $Q$. Let $H$ be the set of atoms of $M$. Suppose that $Q\subseteq H$. Set $H_1=Q$ and $M_1=M$. Let $M_2=\langle H\setminus H_1\rangle$. If $M_2$ is inside factorial with base $Q_2\subseteq H\setminus H_1$, then set $H_2=Q_2$ and $M_3=\langle H\setminus (H_1\cup H_2)\rangle$. If this process stops in a finite number of steps, we would end up with stratification $H_1\cup \dots \cup H_k$ of $H$. If such stratification exists, then we say that $M$ is a \emph{stratified inside factorial} monoid. This situation inspires the following results.
	
\begin{lemma}\label{lem:ap-DM}
	Let $M$ be a root-closed inside factorial monoid. Suppose that $H=H_1\cup \dots \cup H_k$ is a stratification of its set of atoms such that the monoid $\langle H_i\cup\cdots\cup H_k\rangle$ is inside factorial with base $H_i$. Then, for all $i\in \{1,\dots,k-1\}$, $\Ap(M,H_1\cup \dots \cup H_i)\subseteq D_{M_i}(H_i)$.
\end{lemma}
\begin{proof}
	Let $x\in \Ap(M,H_1\cup \dots \cup H_i)$. This, in particular implies that $x\in \langle H_{i+1}\cup \dots \cup H_k\rangle=M_{i+1}\subseteq M_i$. By hypothesis $M_i$ is an inside factorial monoid with basis $H_i$.  Assume that there exists $h\in H_i$ such that the coordinate of $x$ corresponding to $h$ is greater than or equal to one. This means that there exists a positive integer $n$ and nonnegative integers $\lambda_q$ such that $nx=\sum_{q\in H_i}\lambda_q q$, and $\lambda_h/n\ge 1$, whence $\lambda_h\ge n$. This implies that $n(x-h)=nx-nh=(\lambda_h-n)h+\sum_{q\in H_i,q\neq h}\lambda_q q\in M$. As $M$ is root-closed, we deduce that $x-h\in M$, contradicting that $x\in \Ap(M,H_1\cup \dots \cup H_i)$.
\end{proof}

With this we can prove a generalization of Theorem~\ref{th:stratification-full-N2}. 

\begin{theorem}\label{th:For-Stratification}
	Let $M$ be an atomic monoid whose set of atoms admits a stratification $H=H_1\cup\cdots\cup H_k$ such that the monoid $\langle H_i\cup\cdots\cup H_k\rangle$ is inside factorial with base $H_i$.
	Then each $x\in M$ has a representation of the form $x=h_1+h_2+\cdots +h_k$ with
	\begin{enumerate}
		\item $h_i\in\langle H_i\rangle$ for all $i\in \{1,\dots, k\}$,
		\item $h_i+\cdots +h_k\in \Ap(M,H_1\cup\dots\cup H_{i-1})$ for all $i\in\{2,\ldots k\}$.
	\end{enumerate}
	Moreover, if $M$ is root-closed, then the above representation is unique.
\end{theorem}
\begin{proof}
	The existence goes as in Theorem~\ref{uniquerep2}. 

	For the uniqueness, assume that $h_1+\dots+h_k=h_1'+\dots+h_k'$ with $h_i$ and $h_i'$ as in the statement. We show by induction over $i$ that $h_i+\dots+h_k= h_i'+\dots+ h_k'$ implies that $h_{i+1}+\dots+h_k= h_{i+1}'+\dots+ h_k'$. For $i=1$, Theorem~\ref{th:decomposition-apery-coprime} applied to $M=M_1$ and $Q=H_1$, asserts that $h_1=h_1'$ and $h_2+\dots+h_k=h_2'+\dots+h_k'$. Suppose now that $h_i+\dots+h_k=h_i'+\dots+h_k'$. From (1) and (2), $h_i,h_i'\in \langle H_i\rangle$ and $h_{i+1}+\dots+h_k, h_{i+1}'+\dots+h_k' \in \Ap(M,H_1\cup \dots \cup H_i)$. By Lemma~\ref{lem:ap-DM}, $h_{i+1}+\dots+h_k, h_{i+1}'+\dots+h_k'\in D_{M_i}(H_i)$. Now we apply Lemma~\ref{lem:unique-diamond-inside} to $M_i$, $H_i$, $g=h_i-h_i'$, $w=h_{i+1}+\dots+h_k$ and $w'=h_{i+1}'+\dots+h_k'$, obtaining $g=0$ and $w=w'$.
\end{proof}

The conditions imposed for uniqueness in the last two theorems are a bit fragile, and in general one might not expect to hold for dimension larger than two as we can see in the next Example \ref{counterexample}. But we present before a couple of results which have a practical sense.

\begin{lemma}\label{lem:Belong-Apery}
	Let $M$ be a root-closed monoid and $A$ an arbritrary subset of $M$. Then $x\in \Ap(M,A)$ if and only if for each $a\in A$ and $m,n\in \mathbb N\setminus \{0\}$ it holds that $ma\leq_M nx$ implies $m<n$. 
\end{lemma}
\begin{proof}
Let $x\in \Ap(M,A)$ and suppose that $ma\leq_M nx$ for some $a\in A$ with $m\geq n$. Then $na\leq_Mma\leq_Mnx$ and, hence, $n(x-a)\in M$. Since $M$ is root-closed, we get $x-a\in M$ which contradicts $x\in\Ap(M,A)$. Thus, we must have $m<n$.

The converse implication holds even for any $M$. Assuming that $ma\leq_Mnx$ implies $m<n$ we cannot have $a\leq_Mx$ for some $a\in A$ (this would imply $1<1$), and therefore $x\notin M+A$, that is, $x\in \Ap(M,A)$.	
\end{proof}

\begin{corollary}
	Let $M$ be  as in Theorem \ref{th:For-Stratification} and root-closed. Then $H_i$ is the set of all strong atoms of $\langle H_i\cup\cdots\cup H_k\rangle$ for all $i\in\{1,\ldots, k\}$ and condition (2) is equivalent to 
	\begin{itemize}
		\item[(2')] The extraction grades $\lambda_M(h,h_i+\cdots+h_k)<1$ for all $h\in\langle H_1\cup\cdots\cup H_{i-1}\rangle$ and all $i\in\{2,\ldots ,k\}$.
	\end{itemize}
\end{corollary}
\begin{proof}
	For the first part, take $i$ fixed and let $M_i$ be the inside factorial monoid with basis $H_i$ given by the stratification of $H$.
	
	We show first that each $q\in H_i$ is a pure atom of $M_i$. Since $H_i\subseteq H$, we have that $q$ is an atom of $M$ and therefore an atom of $M_i$. Suppose that there is some $x\in M_i\setminus \{0\}$ with $x\leq_{M_i}nq_0$ for some $q_0\in H_i$. As $x\in M_i$ is inside factorial with base $H_i$, we can write $mx=\sum_{q\in H_i} x_qq$ for some $m\in \mathbb N\setminus \{0\}$ and some $x_q\in \mathbb{N}$. We can easily deduce that $mx=\sum_{q\in H_i} x_qq\leq_{M_i} mnq_0$. So there exists $u\in M_i$ such that $\sum_{q\in H_i} x_qq+u= mnq_0$. As $u\in M_i$, we can write $m'u=\sum_{q\in H_i} u_qq$ for some $m'\in \mathbb N\setminus \{0\}$ and some $u_q\in \mathbb{N}$. We obtain $m'\sum_{q\in H_i} x_qq+\sum_{q\in H_i} u_qq= m'mnq_0$. As the element in $H_i$ are $\mathbb{Q}$-linearly independent, we have $m'x_q+u_q=0$ for all $q\in H_i\setminus\{q_0\}$. Observe that $m'x_q+u_q=0$ forces $x_q=u_q=0$, since both are nonnegative integers and $m'\neq 0$. Hence $mx=x_{q_0}q_0$, and so $q_0$ is a pure atom.
	
	Thus, all $q\in H_i$ are pure atoms of $M_i$. Now, let $y\leq_{M_i} n q_0$ for $y\in M_i\setminus\{0\}$, $q_0\in H_i$ and $n$ a positive integer. By definition of $M_i$, $y=\sum_{j=1}^rk_jx_j$ for some $x_j\in H_i\cup \dots \cup H_k$ and $k_j$ positive integers. 
	Therefore $x_j\leq_{M_i} y\leq_{M_i} n q_0$, and as $q_0$ is pure, we deduce that $m_jx_j=n_jq_0$ for some positive integers $m_j$ and $n_j$. If $m_j\leq n_j$, then $n_j(x_j-q_0)=(n_j-m_j)x_j\in M$ and since $M$ is root-closed we get $x_j-q_0\in M$. Therefore $x_j=q_0$, since $x_j$ is an atom of $M$. Analogously, if $n_j\leq m_j$, then $m_j(q_0-x_j)=(m_j-n_j)q_0\in M$, which yields $q_0-x_j\in M$, and thus $x_j=q_0$, since $q_0$ is an atom of $M$. Thus, we arrive at $y=(\sum_{j=i}^r k_j)q_0$. This proves that all $q\in H_i$ are strong atoms of $M_i$.
	
	We now show that every strong atom of $M_i$ must be in $H_i$. Let $a$ be a strong atom of $M_i$. There exists $m\in \mathbb N\setminus\{0\}$ such that $ma=\sum_{q\in H_i}a_qq$ for some nonnegative integers $a_q$. Therefore exists some $q_0$ such that $q_0\leq_{M_i} ma$, and since $a$ is a strong atom we must have that $q_0=n a$ with $n$ a positive integer. Since $q_0$ is an atom of $M_i$, we deduce that $a=q_0\in H_i$.
	
   The equivalence of (2) and (2') follows from Lemma \ref{lem:Belong-Apery} taking $A=H_1\cup\cdots\cup H_{i-1}$. Let $x\in M$ and $a\in A$. The condition that $ma\leq_Mnx$ implies $m<n$ is equivalent to $\lambda_M(a,x)<1$ (being inside factorial, $M$ is an extraction monoid). By Lemma~\ref{lem:Belong-Apery}, $x=h_i+\cdots +h_k\in\Ap(M,A)$ if and only if $\lambda_M(a,x)<1$ for all $a\in A$ which amounts to condition (2').
\end{proof}

The above corollary is helpful in doing computations for concrete cases. In a first step one has to find iteratively the strong atoms of the various monoids. In a second step, to check the inequalities in conditions (2'), one has to calculate the extraction grade $\lambda_M$ for the given monoid $M$, as we did in Example \ref{ex:Elliot-lambda}.

We finish by giving two more examples. The first one shows that the stratification proposed in our motivation section cannot, in general, be performed in higher dimensions. The second provides an example in arbitrary dimension.

\begin{example}\label{counterexample}
	Consider the set $M$ of nonnegative integer solutions of the equation $4x+5y+8z \equiv 0\pmod{11}$, whose Hilbert basis  is 
	\begin{multline*}
		H=\{(0,0,11),(0,11,0),(11,0,0),(0,1,9),(0,5,1),(1,0,5),(9,0,1),(1,8,0),(7,1,0),(0,2,7),\\ 
		(0,4,3),(7,0,2),(3,0,4),(3,2,0),(2,5,0), (0,3,5),(5,0,3),(5,1,1),(1,1,3),(1,2,1),(3,2,1)\}
	\end{multline*}
	(this computation can be performed, for instance, with the \texttt{GAP} \cite{gap} package \texttt{numericalsgps} \cite{numericalsgps}).
	Notice that $M$ is isomorphic as a monoid to the set of nonnegative integer solutions of the equation $4x+5y+8z=11t$ (see for instance \cite[Lemma~12]{GSKL}).

	For the first step of the stratification, we consider $H_1= \{(0,0,11),(0,11,0),(11,0,0)\}$ and the elements of the Ap\'ery set $\operatorname{Ap}(M,H_1)$ can be written in a table as follows:\\
	\resizebox{\hsize}{!}{$\begin{array}{ccccccccccc}
			(0,0,0), &  \mathbf{(1,0,5)}, & (2,0,10), &  \mathbf{(3,0,4)}, & (4,0,9), &  \mathbf{(5,0,3)}, &   (6,0,8), &  \mathbf{(7,0,2)}, &  (8,0,7), &  \mathbf{(9,0,1)}, &  (10,0,6), \\
			\mathbf{(0,1,9)}, &  \mathbf{(1,1,3)}, &  (2,1,8), & \mathbf{(3,1,2)}, & (4,1,7), &  \mathbf{(5,1,1)}, &   (6,1,6), &  \mathbf{(7,1,0)}, &  (8,1,5), & (9,1,10), &  (10,1,4), \\
			\mathbf{(0,2,7)}, &  \mathbf{(1,2,1)}, &  (2,2,6), &  \mathbf{(3,2,0)}, &	(4,2,5), & (5,2,10), &   (6,2,4), &  (7,2,9), &  (8,2,3), &	(9,2,8), &  (10,2,2), \\
			\mathbf{(0,3,5)}, & (1,3,10), &  (2,3,4), &  (3,3,9), &	(4,3,3), &  (5,3,8), &   (6,3,2), &  (7,3,7), &	 (8,3,1), &	(9,3,6), &  (10,3,0), \\
			\mathbf{(0,4,3)}, &  (1,4,8), &  (2,4,2), &  (3,4,7), &	(4,4,1), &  (5,4,6), &   (6,4,0), &  (7,4,5), & (8,4,10), &	(9,4,4), &  (10,4,9), \\
			\mathbf{(0,5,1)}, &  (1,5,6), &  \mathbf{(2,5,0)}, &  (3,5,5), &	(4,5,10)&  (5,5,4), &   (6,5,9), &  (7,5,3), &	 (8,5,8), &	(9,5,2), &  (10,5,7), \\
			(0,6,10), &  (1,6,4), &  (2,6,9), &  (3,6,3), &	(4,6,8), &  (5,6,2), &   (6,6,7), &  (7,6,1), &	 (8,6,6), &	(9,6,0), &  (10,6,5), \\
			(0,7,8), &  (1,7,2), &  (2,7,7), &  (3,7,1), &	(4,7,6), &  (5,7,0), &   (6,7,5), & (7,7,10), &	 (8,7,4), &	(9,7,9), &  (10,7,3), \\
			(0,8,6), &  \mathbf{(1,8,0)}, &  (2,8,5), & (3,8,10), &	(4,8,4), &  (5,8,9), &   (6,8,3), &  (7,8,8), &	 (8,8,2), &	(9,8,7), &  (10,8,1), \\
			(0,9,4), &  (1,9,9), &  (2,9,3), &  (3,9,8), & (4,9,2), &  (5,9,7), &   (6,9,1), &  (7,9,6), &	 (8,9,0), &	(9,9,5), & (10,9,10), \\
			(0,10,2), & (1,10,7), & (2,10,1), & (3,10,6), &	(4,10,0),& (5,10,5), & (6,10,10), & (7,10,4), &	(8,10,9), & (9,10,3),	& (10,10,8). 
		\end{array}$}
	where, we wrote in bold face the elements in the Hilbert basis.
	
	Is not difficult to see that the strong atoms in $\langle H\setminus H_1\rangle$ are $$H_2=\{(0,1,9),(0,5,1),(1,0,5),(9,0,1),(1,8,0),(7,1,0)\}.$$ The other elements in the $\operatorname{Ap}(M,H_1)$ can be written as:
	\[
	\begin{array}{rcl}
		4(0,2,7) & = & 3(0,1,9)+1(0,5,1), \\
		2(0,3,5) & = & 1(0,1,9)+1(0,5,1), \\
		4(0,4,3) & = & 1(0,1,9)+3(0,5,1), \\
		45(1,1,3) & = & 15(0,1,9)+3(1,8,0)+6(7,1,0), \\
		60(1,2,1) & = & 5(0,1,9)+15(0,5,1)+4(1,8,0)+8(7,1,0), \\
		5(2,5,0) & = & 3(1,8,0)+1(7,1,0), \\
		4(3,0,4) & = & 3(1,0,5)+1(9,0,1), \\
		40(3,1,2) & = & 4(1,8,0)+8(7,1,0)+15(1,0,5)+5(9,0,1), \\
		3(3,2,0) & = & 1(2,5,0)+1(7,1,0), \\
		2(5,0,3) & = & 1(1,0,5)+1(9,0,1), \\
		24(5,1,1) & = & 4(2,5,0)+4(7,1,0)+3(1,0,5)+9(9,0,1), \\
		4(7,0,2) & = & 1(1,0,5)+3(9,0,1).
	\end{array}
	\]
	For the next step we obtain $H_3=\{(0,2,7),(0,4,3),(7,0,2),(3,0,4),(3,2,0),(2,5,0)\}$ as the set of strong atoms in $\langle H \setminus (H_1\cup H_2)\rangle$. The rest of the elements in the Hilbert basis of $M$ can be written as:
	\[
	\begin{array}{rcl}
		2(0,3,5) & = & 1(0,2,7)+1(0,4,3), \\
		6(1,1,3) & = & 2(3,0,4)+1(0,2,7)+1(0,4,3), \\
		3(1,2,1) & = & 2(3,2,0)+1(0,4,3), \\
		2(3,1,2) & = & 1(3,2,0)+1(3,0,4), \\ 
		2(5,0,3) & = & 1(7,0,2)+1(3,0,4), \\
		2(5,1,1) & =& 1(7,0,2)+1(3,2,0). 
	\end{array}
	\]
	Finally, the last step is to consider $H_4=\{(0,3,5),(5,0,3),(5,1,1),(1,1,3),(1,2,1)\}$ as the set of strong atoms  of $\langle H\setminus (H_1\cup H_2\cup H_3)\rangle$. But now we have two different ways to write the last element $(3,1,2)$ as a rational combination of the elements in $H_4$:
	$2(3,1,2)=(5,1,1)+(1,1,3)$  or  $2(3,1,2)= (5,0,3)+(1,2,1)$, which comes from  relation  $(5,1,1)+(1,1,3)=(5,0,3)+(1,2,1)$ between elements in $H_4$.
\end{example}

\begin{example} We provide an example inspired by \cite[Lemma~2.6]{jafari}.
	\begin{enumerate}[(1)]
		\item Let $M$ be an inside factorial monoid with base $Q$ and set of atoms $H=Q\cup\{d\}$, $d\notin Q$. We see next that this partition is a stratification of $H$ that satisfies Condition~(S3). Moreover, each $x\in M$ has a unique representation $x=\sum_{q\in Q}\lambda_q q+\lambda d$ with $\lambda,\lambda_q\in\mathbb N$ and $\lambda<\mu(d)=\min\{m\in \mathbb N\setminus\{0\}\mid md\in\langle Q\rangle\}$.
		
		It is clear that $H=H_1\cup H_2$ with $H_1=Q$ and $H_2=\{d\}$ is a stratification of $H$. By Corollary~\ref{cor:decoposition-apery-inside} we have $M=\Ap(M,Q)+\langle Q\rangle$. Condition~(S3) for $H_2$ holds trivially. For $H_1$ we proceed as follows. Let $q,q'\in Q$, $q\neq q'$ and $x+q'=y+q$ with $x,y\in M$. For every $w\in \Ap(M,Q)$ we must have $w=\lambda d$, $\lambda\in \mathbb N$, and the definition of $\mu=\mu(d)$ forces $\lambda<\mu$. Let $x=h+\lambda d$, $y=h'+\lambda' d$ with $h,h'\in \langle Q\rangle$ and  $\lambda,\lambda'< \mu$. It follows that $h+q'+\lambda d=h'+q+\lambda' d$. Without restriction, assume that $\lambda'\leq \lambda$  and, hence, $h'+q-(h+q')=(\lambda-\lambda')d\in M$. For any inside factorial monoid with base $Q$ it holds for $a,b\in\langle Q\rangle$ that $a\leq_M b$ is equivalent to $a\leq_{\langle Q\rangle} b$. Thus, we obtain $(\lambda-\lambda')d\in\langle Q\rangle$.
		Since $0\leq \lambda-\lambda'< \mu(d)$ we must have $\lambda=\lambda'$. This yields $h+q'=h'+q$. As $\langle Q\rangle$ is factorial and $q\neq q'$, we must have $q\leq_M h$ and, hence, $q\leq_M x$. This proves (S3) for $H_1$. Applying Theorem~\ref{uniquerep2} yields a unique representation $x=h_1+h_2$, $h_1=\sum_{q\in Q}\lambda_q q\in\langle Q\rangle$ and $h_2\in\Ap(M,Q)$; $h_2=\lambda d$ for some nonnegative integer $\lambda < \mu(d)$.

		\item If $H=Q\cup\{d_1,\ldots,d_r\}$ with $d_i\notin Q$ and $m_id_i=n_id_1$ for some $m_i,n_i\in\mathbb N\setminus\{0\}$, $i\in\{1,\dots, r\}$, then $H_1=Q$, $H_{i+1}=\{d_i\}$, for $i\in \{1,\dots,r\}$ gives a stratification of $H$ fulfilling Condition~(S3).

		Clearly, $H=H_1\cup\cdots\cup H_{r+1}$ with $H_1=Q$, $H_{i+1}=\{d_i\}$, $i\in\{1,\dots, r\}$ is a stratification of $H$. Condition~(S3) for $H_{i+1}$, $i\in \{1,\dots,r\}$, holds trivially. From Corollary~\ref{cor:decoposition-apery-inside} we have that $M=\Ap(M,Q)+\langle Q\rangle$ and we shall show that this sum is direct. By Theorem~\ref{th:decomposition-apery-coprime} this is equivalent to Condition~(S3) for $H_1=Q$.
	
		For $x\in M$, $x=h+w$ with $h\in\langle Q\rangle$, $w\in\Ap(M,Q)$. Observe that $w=\sum_{i=1}^rk_id_i$. If we define $m=\prod_{i=1}^rm_i=m_j\cdot m_j'$, $j\in\{1,\dots, r\}$, we obtain $mw=\lambda d_1$, where $\lambda=\sum_{i=1}^rk_im_i'n_i\in \mathbb N$.
	
		Now, suppose that $x=h+w=h'+w'$, for some $h'\in\langle Q\rangle$ and $w'\in\Ap(M,Q)$. Arguing as above, we have $mw'=\lambda'd_1$, and consequently $mh+\lambda d_1=mh'+\lambda'd_1$. Suppose without loss of generality that $\lambda'\leq\lambda$, and, hence, $m(h'-h)=(\lambda-\lambda')d_1\in M$. As in part (1), $(h-h')\in\langle Q \rangle$, and since a factorial monoid is root closed, we get $\tilde h=h'-h\in\langle Q\rangle$. From $w-w'=h'-h$ we obtain $\tilde h\leq_M w$, and since $w\in \Ap(M,Q)$ we must have $\tilde h=0$. Therefore $h=h'$ and $w=w'$. 
\end{enumerate}
\end{example}

\end{document}